\documentclass{amsart}
\usepackage{amsmath,amssymb, amsthm, amscd}
\input xy
\xyoption{all}

\newtheorem{theorem}{Theorem}
\newtheorem{corollary}[theorem]{Corollary}

\newtheorem{lemma}[theorem]{Lemma}
\newtheorem{proposition}[theorem]{Proposition}

\newtheorem*{tamesymbol}{Definition}
\newtheorem*{introtheorem}{Theorem}

\newcommand{\GL}{\operatorname{GL}}
\newcommand{\M}{\operatorname{M}}

\begin{document}

\title{The tame symbol and determinants of Toeplitz operators}
\author{Efton Park}
\address{Department of Mathematics, Box 298900, 
Texas Christian University, Fort Worth, TX 76129}
\email{e.park@tcu.edu}

\begin{abstract}
Suppose that $\phi$ and $\psi$ are smooth complex-valued functions on the circle that are invertible, have
winding number zero with respect to the origin, and have meromorphic extensions to an open neighborhood of the
closed unit disk.  Let $T_\phi$ and $T_\psi$ denote the Toeplitz operators with symbols $\phi$ and $\psi$ respectively.
We give an explicit formula for the determinant of $T_\phi T_\psi T_\phi^{-1} T_\psi^{-1}$ in terms of the products of 
the tame symbols of $\phi$ and $\psi$ on the open unit disk.
\end{abstract}
\maketitle

For each continuous complex valued function $\phi$ on the unit circle, let $T_\phi$ denote the 
Toeplitz operator $T_\phi$ on the Hardy space $H^2(S^1)$.  If $\phi$ and $\psi$ are
smooth functions on $S^1$, it is easy to check that the commutator $T_\phi T_\psi - T_\psi T_\phi$ 
is an element of the ideal $\mathcal{L}^1$ of trace class operators on $H^2(S^1)$.  If in addition $\phi$ and $\psi$ are 
invertible and have winding number zero with respect
to the origin, then $T_\phi$ and $T_\psi$ are invertible (\cite{Douglas}, Corollary 7.27).  The multiplicative
commutator $T_\phi T_\psi T_\phi^{-1} T_\psi^{-1} = I + (T_\phi T_\psi - T_\psi T_\phi) T_\phi^{-1} T_\psi^{-1}$
has the form identity plus a trace class operator, and therefore has a well-defined determinant.  There are various
integral formulas for this determinant (\cite{HH}, Proposition 10.1; \cite{Brown1}, Section 6; 
\cite{Kaad}, Theorem 6.2).  There is also a formula (\cite{CP}, Proposition 1) that uses ideas from 
\cite{Deligne} to expresses this determinant in terms of a quantity called the \emph{tame symbol}.
In this paper, we give a relatively elementary proof of the
result in \cite{CP} in cases where $\phi$ and $\psi$ are smooth invertible functions
on the unit circle that have meromorphic extensions to a neighborhood of the closed unit disk.

\begin{tamesymbol}
Suppose that $\phi$ and $\psi$ are restrictions of meromorphic functions (which we also denote $\phi$ and $\psi$)
defined in a neighborhood of the closed unit disk such that neither $\phi$ nor $\psi$ has zeros or poles on the unit circle.  For each point $z$ in the open unit disk $\mathbb{D}$, define 
\[
v(\phi, z) =
 \begin{cases}
\ \ m &\text{ if $\phi$ has a zero of order $m$ at $z$} \\
-m &\text{ if $\phi$ has a pole of order $m$ at $z$} \\
\ \ \ 0 &\text{ if $\phi$ has neither a zero nor a pole at $z$,}
\end{cases}
\]
and similarly define $v(\psi, z)$.  The quantity
\[
\lim_{w \to z} (-1)^{v(\phi, z) v(\psi, z)}\frac{\psi(w)^{v(\phi, z)}}{\phi(w)^{v(\psi, z)}}
\]
is called the \emph{tame symbol} of $\phi$ and $\psi$ at $z$ and is denoted $(\phi, \psi)_z$.  
\end{tamesymbol}

\begin{introtheorem}\label{main}
Suppose that $\phi$ and $\psi$ satisfy the condition stated in the previous definition and 
have winding number zero with respect to the origin.  Then
\[
\det(T_\phi T_\psi T_\phi^{-1} T_\psi^{-1})  = \prod_{z \in \mathbb{D}} (\phi, \psi)_z^{-1}.
\]
\end{introtheorem}

In fact we prove a somewhat more general result that allows us to consider determinants of operators
involving invertible functions $\phi$ and $\psi$ that have nonzero winding numbers.  We accomplish this
by using a construction from algebraic $K$-theory; L. Brown showed in \cite{Brown2} that a close
relationship exists between operator determinants and the connecting map in the algebraic $K$-theory
long exact sequence.  We then use the multiplicative properties of the determinant and the tame symbol
to reduce the problem to cases where we can prove the desired result by direct computation.

We begin by establishing some notation.  Let $C^\mu(S^1)$ denote the algebra of functions
on the unit circle that have a meromorphic extension to a neighborhood of the closed unit disk; we will abuse
notation through the paper by using the same symbol to denote a function in $C^\mu(S^1)$ and its
meromorphic extension. We will write $(C^\mu(S^1))^*$ for the group under multiplication of those functions
in $C^\mu(S^1)$ whose meromorphic extension has no poles or zeros on $S^1$.  

\begin{proposition}\label{properties of pi}
Define $\pi: (C^\mu(S^1))^* \longrightarrow \mathbb{C}^*$ by the formula
\[
\pi(\phi, \psi) = \prod_{z \in \mathbb{D}} (\phi, \psi)_z^{-1}.
\]
Then for all $\phi$, $\phi^\prime$, $\psi$, and $\psi^\prime$ in $(C^\mu(S^1))^*$,
\begin{enumerate}
\item[(i)] $\pi(\phi\phi^\prime, \psi) =  \pi(\phi, \psi) \pi(\phi^\prime, \psi)$;
\item[(ii)] $\pi(\phi, \psi\psi^\prime) =  \pi(\phi, \psi)  \pi(\phi, \psi^\prime)$;
\item[(iii)] $\pi(\psi, \phi) = \left(\pi(\phi, \psi)\right)^{-1}$.
\end{enumerate}
\end{proposition}

\begin{proof}
Obvious from the definition of the tame symbol.
\end{proof}

Set
\[
\mathcal{T}^\mu = \{T_\phi + L : \phi \in C^\mu(S^1), L \in \mathcal{L}^1\}.
\]
Because functions in $C^\mu(S^1)$ are smooth on the unit circle, the set $\mathcal{T}^\mu$ is closed under
multiplication and is a $\mathbb{C}$-algebra.  We have a short exact sequence
\[
\xymatrix@C=30pt
{0 \ar[r]^-{} & \mathcal{L}^1 \ar[r]^-{} & \mathcal{T}^\mu \ar[r]^-{\sigma}  & C^\mu(S^1) \ar[r]^-{} & 0,}
\]
where $\sigma$ is the symbol map defined by $\sigma(T_\phi + L) = \phi$ for every $T_\phi + L$ in
$\mathcal{T}^\mu$.  This exact sequence gives rise to a short exact sequence of matrices
\[
\xymatrix@C=30pt
{0 \ar[r]^-{} & \M(3, \mathcal{L}^1) \ar[r]^-{} & \M(3, \mathcal{T}^\mu \ar[r]^-{\sigma})  & \M(3, C^\mu(S^1)) \ar[r]^-{} & 0}
\]
with $\sigma$ defined entrywise.  

\begin{lemma}\label{existence of lifts}
Let $\phi$ and $\psi$ be elements of $(C^\mu(S^1))^*$ and define matrices
\[
\Phi^{1,2} = \begin{pmatrix} \phi & 0 & 0 \\ 0 & \phi^{-1} & 0 \\ 0 & 0 & 1 \end{pmatrix}, \qquad
\Psi^{1,3} = \begin{pmatrix} \psi & 0 & 0 \\ 0 & 1 & 0 \\ 0 & 0 & \psi^{-1} \end{pmatrix}.
\]
Then there exists matrices $R_\psi$ and $S_\phi$ in $\GL(3, \mathcal{T}^\mu)$ such that
$\sigma(R_\phi) = \Phi^{1,2}$ and $\sigma(S_\psi) = \Psi^{1,3}$.
\end{lemma}

\begin{proof}
Using the recipe from Corollary 2.1.3 in \cite{Rosenberg}, we know that $\Phi^{1,2}$ lifts to a matrix
\[
R_\phi = \begin{pmatrix}
2T_\phi - T_\phi T_{\phi^{-1}}T_\phi & T_\phi T_{\phi^{-1}} - I & 0 \\
I - T_{\phi^{-1}} T_\phi & T_{\phi^{-1}} & 0 \\
0 & 0 & I
\end{pmatrix}
\]
in $\GL(3, \mathcal{T}^\mu)$ with inverse
\[
\begin{pmatrix}
 T_{\phi^{-1}} & I -  T_{\phi^{-1}}T_\phi & 0 \\
 T_\phi  T_{\phi^{-1}} - I & 2T_\phi - T_\phi  T_{\phi^{-1}} T_\phi & 0 \\
 0 & 0 & I
\end{pmatrix}. 
\]
Similarly, for any $\psi$ in $(C^\mu(S^1))^*$, the matrix $\Psi^{1,3}$ lifts to the invertible matrix
\[
S_\psi = \begin{pmatrix}
2T_\psi - T_\psi T_{\psi^{-1}}T_\psi & 0 &T_\psi T_{\psi^{-1}} - I \\
0 & I & 0 \\
I - T_{\psi^{-1}} T_\psi & 0 & T_{\psi^{-1}} 
\end{pmatrix}
\]
with inverse
\[
\begin{pmatrix}
T_{\psi^{-1}} &  0  & I - T_{\psi^{-1}}T_\psi \\
0 & I & 0 \\
T_\psi T_{\psi^{-1}} - I  & 0 &  2T_\psi - T_\psi T_{\psi^{-1}}T_\psi  
\end{pmatrix}.
\]
\end{proof}

There are other choices for lifts of $\Phi^{1,2}$ and $\Psi^{1,3}$, but the next lemma
shows that for our purposes the choice of lift does not matter.

\begin{lemma}\label{delta well defined}
Let $\phi$ and $\psi$ be elements of $(C^\mu(S^1))^*$ and let $R$ and $S$ be elements of 
$\GL(3, \mathcal{T}^\mu)$ such that $\sigma(R) = \Phi^{1,2}$ and $\sigma(S) = \Psi^{1,3}$.
Then $\det(RSR^{-1}S^{-1})$ is independent of the choices of $R$ and $S$.
\end{lemma}

\begin{proof}
Suppose $\widetilde{R}$ is another element of $\GL(3, \mathcal{T}^\mu)$ with the property
that $\sigma(\widetilde{R}) = \Phi^{1,2}$.  Then the multiplicativity and similarity invariance of the
determinant give us the following string of equalities:
\begin{align*}
\det(\widetilde{R}S\widetilde{R}^{-1}S^{-1}) 
&= \det(R^{-1}\widetilde{R}S\widetilde{R}^{-1}S^{-1}R) \\
&= \det(R^{-1}\widetilde{R}) \det(S\widetilde{R}^{-1}S^{-1}R) \\
&= \det(R^{-1}\widetilde{R}) \det(S\widetilde{R}^{-1}S^{-1}RSS^{-1}) \\
&= \det(R^{-1}\widetilde{R}) \det(\widetilde{R}^{-1}S^{-1}RS) \\
&= \det(R^{-1}\widetilde{R}) \det(R\widetilde{R}^{-1}S^{-1}RSR^{-1}) \\
&= \det(R^{-1}\widetilde{R}) \det(R\widetilde{R}^{-1}) \det(S^{-1}RSR^{-1}) \\
&= \det(RR^{-1}\widetilde{R}R^{-1}) \det(R\widetilde{R}^{-1}) \det(SS^{-1}RSR^{-1}S^{-1}) \\
&= \det(\widetilde{R}R^{-1}) \det(R\widetilde{R}^{-1})  \det(RSR^{-1}S^{-1}) \\
&= \det(RSR^{-1}S^{-1}).
\end{align*}
A similar argument shows that if $\widetilde{S}$ is another element of $\GL(3, \mathcal{T}^\mu)$ with the property
that $\sigma(\widetilde{S}) = \Psi^{1,3}$, then $\det(R\widetilde{S}R^{-1}\widetilde{S}^{-1}) = 
\det(RSR^{-1}S^{-1})$.
\end{proof}

In fact more is true: the quantity $\det(RSR^{-1}S^{-1})$ only depends on
the Steinberg symbol $\{\phi, \psi\}$ in the algebraic $K$-theory group $K_2(C^\mu(S^1))$; see 
\cite{Rosenberg}, Proposition 4.4.22.

In light of Lemma \ref{delta well defined}, we can define a map 
$\delta: (C^\mu(S^1))^* \longrightarrow \mathbb{C}^*$ by the formula
\[
\delta(\phi, \psi) = \det(RSR^{-1}S^{-1}),
\]
where $R$ and $S$ are any elements of $\GL(3, \mathcal{T}^\mu)$ with the feature that
$\sigma(R) = \Phi^{1,2}$ and $\sigma(S) = \Psi^{1,3}$.

\begin{proposition}\label{properties of delta}
For all $\phi$, $\phi^\prime$, $\psi$, and $\psi^\prime$ in $(C^\mu(S^1))^*$,
\begin{enumerate}
\item[(i)] $\delta(\phi\phi^\prime, \psi) =  \delta(\phi, \psi) \delta(\phi^\prime, \psi)$;
\item[(ii)] $\delta(\phi, \psi\psi^\prime) =  \delta(\phi, \psi)  \delta(\phi, \psi^\prime)$;
\item[(iii)] $\delta(\psi, \phi) = \left(\delta(\phi, \psi)\right)^{-1}$.
\end{enumerate}
\end{proposition}

\begin{proof}
Let $R^\prime$ be a matrix in $\GL(3, \mathcal{T}^\mu)$ whose image under $\sigma$ is
\[
\begin{pmatrix} \phi^\prime & 0 & 0 \\ 0 & (\phi^\prime)^{-1} & 0 \\ 0 & 0 & 1 \end{pmatrix}.
\]
Then
\begin{align*}
\delta(\phi\phi^\prime, \psi) &= \det(RR^\prime S(RR^\prime)^{-1}S^{-1}) \\
&= \det(RR^\prime S(R^\prime)^{-1}R^{-1}S^{-1}) \\
&= \det \left((RR^\prime S(R^\prime)^{-1})(S^{-1}R^{-1}RS)(R^{-1}S^{-1})\right) \\
&= \det(RR^\prime S(R^\prime)^{-1}S^{-1}R^{-1}) \det(RSR^{-1}S^{-1}) \\
&= \det(R^\prime S(R^\prime)^{-1}S^{-1}) \det(RSR^{-1}S^{-1}) \\
&= \det(RSR^{-1}S^{-1})  \det(R^\prime S(R^\prime)^{-1}S^{-1}) \\
&= \delta(\phi, \psi) \delta(\phi^\prime, \psi).
\end{align*}
This proves (i), and a similar argument yields (ii).  To establish (iii), let $R$ and $S$ be lifts to
$\GL(3, \mathcal{T}^\mu)$ of $\Phi^{1, 2}$ and $\Psi^{1, 3}$ respectively.  Define $J$ in
$\GL(3, \mathcal{T}^\mu)$ as 
\[
J = \begin{pmatrix} I & 0 & 0 \\ 0 & 0 & I \\ 0 & I & 0 \end{pmatrix}.
\]
Then $JRJ^{-1}$ and $JSJ^{-1}$ are elements of $\GL(3, \mathcal{T}^\mu)$ with the property that
\[
\sigma(JRJ^{-1}) = \begin{pmatrix} \phi & 0 & 0 \\ 0 & 1 & 0 \\ 0 & 0 & \phi^{-1}\end{pmatrix} \qquad
\text{and} \qquad \sigma(JSJ^{-1}) = \begin{pmatrix} \psi & 0 & 0 \\ 0 & \psi^{-1} & 0 \\ 0 & 0 & 1\end{pmatrix}.
\]
We therefore have 
\begin{align*}
\delta(\psi, \phi) &= \det\left((JSJ^{-1})(JRJ^{-1})(JSJ^{-1})^{-1}(JRJ^{-1})^{-1}\right) \\
&= \det(JSRS^{-1}R^{-1}J^{-1}) \\
&= \det(SRS^{-1}R^{-1}) \\
&= \left(\det(RSR^{-1}S^{-1}\right))^{-1} \\
&= \left(\delta(\phi, \psi)\right)^{-1}.
\end{align*}
\end{proof}

The next step is to compute $\delta(\phi, \psi)$ for some special choices of $\phi$ and $\psi$.  First we prove two lemmas
that will simplify some of our computations.

\begin{lemma}\label{inverse of holomorphic}
Let $\phi$ be a holomorphic function in a neighborhood of the closed unit disk, and suppose that
$\phi$ has no zeros on the closed unit disk.  Then $T_\phi$ is invertible,
and $T_\phi^{-1} = T_{\phi^{-1}}$.
\end{lemma}

\begin{proof}
The Fourier expansion of $\phi$ on the circle does not contain any negative powers of $z$, and because $\phi$ has no
zeros on the closed unit disk, the Fourier expansion of $\phi^{-1}$ on the circle also cannot contain any
negative powers of $z$.  A direct computation then shows that $T_\phi T_{\phi^{-1}} = T_{\phi^{-1}} T_\phi = I$.
\end{proof}

\begin{lemma}\label{inverse of linear}
Let $\alpha$ be a complex number with modulus strictly less than $1$.  Then $T_{(z - \alpha)^{-1}}T_{z - \alpha} = I$,
and for $f = \sum_{k=0}^\infty c_kz^k$ in the Hardy space $H^2(S^1)$, 
\[
T_{z - \alpha}T_{(z - \alpha)^{-1}}f = -\sum_{n=1}^\infty\alpha^nc_n + \sum_{k=1}^\infty c_kz^k.
\]
\end{lemma}

\begin{proof}
Because $|\alpha| < 1$,
\[
\frac{1}{z - \alpha} = \frac{\frac{1}{z}}{1 - \frac{\alpha}{z}} = 
\sum_{n=1}^\infty \alpha^{n-1}z^{-n},
\]
and thus
\[
T_{(z - \alpha)^{-1}} = \sum_{n=1}^\infty \alpha^{n-1}T_{z^{-n}}.
\]
The fact that $T_{z^{-1}}T_z = I$ implies the first equation in the statement of the lemma.
Let $P$ denote the orthogonal projection of $L^2(S^1)$ onto $H^2(S^1)$. 
\begin{align*}
T_{(z - \alpha)^{-1}}f &= \left(\sum_{n=1}^\infty \alpha^{n-1}T_{z^{-n}}\right)\left(\sum_{k=0}^\infty c_kz^k\right) \\
&= \sum_{n=1}^\infty \sum_{k=0}^\infty \alpha^{n-1}c_k P z^{k - n} \\
&= \sum_{n=1}^\infty \sum_{k=n}^\infty \alpha^{n-1}c_k z^{k - n}.
\end{align*}
Set $\ell = k - n$.  Then we obtain
\[
T_{(z - \alpha)^{-1}}f =  \sum_{n=1}^\infty \sum_{\ell=0}^\infty \alpha^{n-1}c_{\ell+n} z^\ell = 
\sum_{\ell=0}^\infty\left(\sum_{n=1}^\infty \alpha^{n-1}c_{\ell+n}\right)z^\ell,
\]
and so
\begin{align*}
T_{z - \alpha}T_{(z - \alpha)^{-1}}f &= \left(T_z - \alpha I\right)
\sum_{\ell=0}^\infty\left(\sum_{n=1}^\infty \alpha^{n-1}c_{\ell+n}\right)z^\ell \\
&= \sum_{\ell=0}^\infty\left(\sum_{n=1}^\infty \alpha^{n-1}c_{\ell+n}\right)z^{\ell+1} -
\sum_{\ell=0}^\infty\left(\sum_{n=1}^\infty \alpha^n c_{\ell+n}\right)z^\ell \\
&= \sum_{\ell=1}^\infty\left(\sum_{n=1}^\infty \alpha^{n-1}c_{\ell+n-1}\right)z^\ell -
\sum_{\ell=0}^\infty\left(\sum_{n=1}^\infty \alpha^n c_{\ell+n}\right)z^\ell \\
&= -\sum_{n=1}^\infty \alpha^n c_n + 
\sum_{\ell=1}^\infty\left(\sum_{n=1}^\infty \left(\alpha^{n-1}c_{\ell+n-1} - \alpha^n c_{\ell+n}\right)\right)z^\ell.
\end{align*}
But for $\ell \geq 1$,
\begin{align*}
\sum_{n=1}^\infty\left(\alpha^{n-1}c_{\ell+n-1} - \alpha^n c_{\ell+n}\right) &=
\sum_{n=1}^\infty\alpha^{n-1}c_{\ell+n-1} - \sum_{n=1}^\infty\alpha^n c_{\ell+n} \\
&= \sum_{n=0}^\infty\alpha^{n}c_{\ell+n} - \sum_{n=1}^\infty\alpha^n c_{\ell+n} \\
&= c_\ell,
\end{align*}
whence
\[
T_{z - \alpha}T_{(z - \alpha)^{-1}}f = -\sum_{n=1}^\infty \alpha^nc_n + \sum_{\ell=1}^\infty c_\ell z^\ell = 
-\sum_{n=1}^\infty \alpha^nc_n + \sum_{k=1}^\infty c_k z^k.
\]
\end{proof}

\begin{proposition}\label{both holomorphic}
Suppose $\phi$ and $\psi$ are holomorphic functions in a neighborhood of the unit disk and have no zeros on
the closed unit disk.  Then $\delta(\phi, \psi) = \pi(\phi, \psi)$.
\end{proposition}

\begin{proof}
We see immediately from the definition of $\pi$ that $\pi(\phi, \psi) = 1$.  Lemma \ref{inverse of holomorphic}
implies that the matrices $R_\phi$ and $S_\psi$ from Lemma \ref{existence of lifts} simplify to
\[
R_\phi = \begin{pmatrix} T_\phi & 0 & 0 \\ 0 & {T_\phi}^{-1} & 0 \\ 0 & 0 & I \end{pmatrix} \qquad
\text{and} \qquad
S_\psi = \begin{pmatrix} T_\psi & 0 & 0 \\ 0 & I & 0 \\ 0 & 0 & {T_\psi}^{-1} \end{pmatrix}.
\]
The Toeplitz operators $T_\phi$ and $T_\psi$ commute, and thus 
\[
\delta(\phi, \psi) = \det(R_\phi S_\psi R_\phi^{-1} S_\psi^{-1}) = \det I = 1.
\]
\end{proof}

\begin{proposition}\label{holomorphic and linear}
Suppose $\phi$ is a holomorphic functions in a neighborhood of the unit disk and has no zeros on
the closed unit disk, and suppose $|\alpha| < 1$.  Then 
$\delta(\phi, z - \alpha) = \pi(\phi, z - \alpha)$.
\end{proposition}

\begin{proof} 
If $\phi$ is constant, the proposition follows immediately, so suppose otherwise.
The function $\phi$ has no poles or zeros in the disk, while $\psi$ has a simple zero at $\alpha$.  The
tame symbol at $\alpha$ is
\[
(\phi(z), z - \alpha)_\alpha = \lim_{w \to \alpha} (-1)^{(0)(1)}\frac{(w - \alpha)^0}{\phi(w)^1} =
\frac{1}{\phi(\alpha)},
\]
and thus $\pi(\phi, \psi) = \phi(\alpha)$.

To simplify notation, set $\psi(z) = z - \alpha$.  Lemmas \ref{existence of lifts}, \ref{inverse of holomorphic},
and \ref{inverse of linear} give us
\[
R_\phi = \begin{pmatrix}
T_\phi  & 0 & 0 \\
0 & {T_\phi}^{-1} & 0 \\
0 & 0 & I
\end{pmatrix} \qquad \text{and} \qquad
S_\psi= \begin{pmatrix}
T_\psi & 0 &T_\psi T_{\psi^{-1}} - I \\
0 & I & 0 \\
0 & 0 & T_{\psi^{-1}} \\
\end{pmatrix},
\]
and computing $R_\phi S_\phi R_\phi^{-1} S_\psi^{-1}$ yields
\begin{align*}
\delta(\phi, \psi) &= 
\det \begin{pmatrix}
T_\phi T_\psi T_{\phi^{-1}} T_{\psi^{-1}} + T_\phi(T_\psi T_{\psi^{-1}} - I)^2 & 0 & 0 \\
0 & I & 0 \\
0 & 0 & I
\end{pmatrix} \\
&= \det\left(T_\phi T_\psi T_{\phi^{-1}} T_{\psi^{-1}} + T_\phi(T_\psi T_{\psi^{-1}} - I)^2\right).
\end{align*}
Note that
\[
(T_\psi T_{\psi^{-1}} - I)^2 = T_\psi T_{\psi^{-1}}T_\psi T_{\psi^{-1}} - 2T_\psi T_{\psi^{-1}} + I
= I - T_\psi T_{\psi^{-1}}.
\]
In addition, the operators $T_\psi$ and $T_{\phi^{-1}}$ commute, so 
\[
\delta(\phi, \psi) = \det\left(T_\psi T_{\psi^{-1}} + T_\phi(I - T_\psi T_{\psi^{-1}})\right).
\]
Write 
\[
\phi(z) = \sum_{k=0}^\infty a_kz^k. 
\]
Then 
\[
T_\phi =  \sum_{k=0}^\infty a_kT_{z^k}.  
\]
Take $f = \sum_{k=0}^\infty c_kz^k$ in the Hardy space $H^2(S^1)$.  Then a straightforward
computation using Lemma \ref{inverse of linear} yields
\[
\left(T_\psi T_{\psi^{-1}} + T_\phi(I - T_\psi T_{\psi^{-1}})\right)\! f = 
-\sum_{n=1}^\infty \alpha^nc_n + \sum_{k=1}^\infty c_kz^k + 
\sum_{k=0}^\infty a_k\left(\sum_{n=0}^\infty \alpha^n c_n\right)\! z^k.
\]
Suppose that $f$ is an eigenvector with eigenvalue $\lambda$.  Then for all $k \geq 1$ we have
\[
\lambda c_k = c_k + a_k\sum_{n=0}^\infty \alpha^n c_n,
\]
and therefore for any $j, k \geq 1$, we obtain the system of equations
\begin{align*}
\lambda a_j c_k &= a_j c_k + a_j a_k\sum_{n=0}^\infty \alpha^n c_n \\
\lambda a_k c_j &= a_k c_j + a_j a_k\sum_{n=0}^\infty \alpha^n c_n \\
\end{align*}
which gives us 
\[
\lambda(a_j c_k - a_k c_j) = a_j c_k - a_k c_j.
\]
Therefore, if $\lambda \neq 1$, we must have $a_j c_k - a_k c_j = 0$ for all $j, k \geq 1$.  For the
remainder of the proof, fix a value of $j$ with the property that $a_j \neq 0$.  Then
\[
c_k = \left(\frac{c_j}{a_j}\right)a_k
\]
for all $k \geq 1$, and so if our operator has an eigenvector with eigenvalue $\lambda \neq 1$, it must have
an eigenvector with the feature that $c_k = a_k$ for all $k \geq 1$.  If we plug this information into
the equation 
\[
\lambda c_j = c_j + a_j\sum_{n=0}^\infty \alpha^n c_n
\]
we obtain 
\[
\lambda a_j = a_j + a_j\left(c_0 + \sum_{n=1}^\infty \alpha^n a_n\right),
\]
whence
\[
\lambda = 1 + c_0 + \sum_{n=1}^\infty \alpha^n a_n.
\]
Looking at the constant term of $f$, we also have the equations
\begin{align*}
\lambda c_0 &= -\sum_{n=1}^\infty \alpha^n c_n + a_0\sum_{n=0}^\infty \alpha^n c_n \\
&= -\sum_{n=1}^\infty \alpha^n c_n + a_0c_0 + \sum_{n=1}^\infty \alpha^n a_n, \\
\end{align*}
which, when combined with our expression for $\lambda$, give us
\[
\left(1 + c_0 + \sum_{n=1}^\infty \alpha^n a_n \right)c_0 = 
-\sum_{n=1}^\infty \alpha^n a_n + a_0c_0 + a_0\sum_{n=1}^\infty \alpha^na_n.
\]
This last equation can be rewritten as
\[
\left(c_0 + 1 - a_0\right)\left(c_0 + \sum_{n=1}^\infty \alpha^na_n\right) = 0.
\]
If $c_0 = - \sum_{n=1}^\infty \alpha^na_n$, then $\lambda = 1$.  If $c_0 = a_0 - 1$, then
\[
\lambda = 1 + a_0 - 1 + \sum_{n=1}^\infty \alpha^n a_n = \phi(\alpha),
\]
from which we infer that $\delta(\phi, z - \alpha) = \phi(\alpha)$.
\end{proof}

\begin{proposition}\label{alpha < 1 beta < 1}
Let $\phi(z) = z - \alpha$ and $\psi(z) = z - \beta$ with $|\alpha| < 1$ and $|\beta| < 1$.  Then
$\delta(\phi, \psi) = \pi(\phi, \psi)$.
\end{proposition}

\begin{proof} 
The functions $\phi$ and $\psi$ have simple zeros at $\alpha$ and $\beta$ respectively.  If 
$\alpha \neq \beta$, then  
\[
(z - \alpha, z - \beta)_\alpha = \lim_{w \to \alpha} (-1)^{(1)(0)}\frac{(w - \beta)^1}{(w - \alpha)^0} =
\alpha - \beta
\]
and
\[
(z - \alpha, z - \beta)_\beta = \lim_{w \to \beta} (-1)^{(0)(1)}\frac{(w - \beta)^0}{(w - \alpha)^1} =
\frac{1}{\beta - \alpha},
\]
whence $\pi(\phi, \psi) = -1$.  On the other hand, if $\alpha = \beta$, 
\[
(z - \alpha, z - \alpha)_\alpha = \lim_{w \to \alpha} (-1)^{(1)(1)}\frac{(w - \alpha)^1}{(w - \alpha)^1} = - 1,
\]
and so $\pi(\phi, \psi) = -1$ regardless if $\alpha$ and $\beta$ are equal or not.

A computation similar to that done in Proposition \ref{holomorphic and linear} shows that
\[
\delta(\phi, \psi) = \det\left(T_\phi T_\psi T_{\phi^{-1}} T_{\psi^{-1}} + 
(I - T_\phi T_{\phi^{-1}})T_{\psi^{-1}} + T_\phi(I - T_\psi T_{\psi^{-1}})\right).  
\]

Take $f = \sum_{k=0}^\infty c_kz^k$ in the Hardy space $H^2(S^1)$.  Then a somewhat involved computation yields
\begin{multline*}
\left(T_\phi T_\psi T_{\phi^{-1}} T_{\psi^{-1}} + 
(I - T_\phi T_{\phi^{-1}})T_{\psi^{-1}} + T_\phi(I - T_\psi T_{\psi^{-1}})\right)\! f = \\
- \alpha c_0 + \sum_{m=1}^\infty \sum_{n=1}^\infty \alpha^m(\beta^n - \beta^{n-1})c_{m+n} 
+ \sum_{n=1}^\infty (\beta^{n-1} - \alpha\beta^n)c_n \\
+ \left(c_0  - \sum_{m=1}^\infty \alpha^m c_{m+1} + \sum_{n=1}^\infty \beta^n c_n
- \sum_{m=1}^\infty \sum_{n=1}^\infty \alpha^{m-1} \beta^n c_{m+n}\right){\hskip -3pt}z 
+ \sum_{k=2}^\infty c_k z^k.
\end{multline*}
Therefore if $f$ is an eigenvector with eigenvalue $\lambda \neq 1$, then $c_k = 0$ for $k \geq 2$.  In this case,
\[
-\alpha c_0 + (1 - \alpha\beta) c_1 + (c_0 + \beta c_1)z = \lambda c_0 + \lambda c_1 z,
\]
which gives us the equations
\begin{align*}
(\lambda + \alpha)c_0 + (\alpha\beta - 1)c_1 &= 0 \\
-c_0 + (\lambda - \beta)c_1 &=0.
\end{align*}
For this system to have a nontrivial solution, we must have
\[
0 = (\lambda + \alpha)(\lambda - \beta) + \alpha\beta - 1 = \lambda^2 + (\alpha - \beta)\lambda - 1.
\]
Thus the product of the roots of the quadratic equation is $-1$, as desired.
\end{proof}

\begin{theorem}\label{big theorem}
For all $\phi$ and $\psi$ in $(C^\mu(S^1))^*$, we have $\delta(\phi, \psi) = \pi(\phi, \psi)$.  
\end{theorem}

\begin{proof}
Every meromorphic function defined on a neighborhood of the closed unit disk can be written as a 
rational function times a function that is holomorphic and has no zeros on the closed unit disk.
Therefore the desired result follows from Lemmas \ref{properties of pi} and \ref{properties of delta},
Propositions \ref{both holomorphic} and \ref{holomorphic and linear}, and the Fundamental Theorem of Algebra.
\end{proof}

We now apply our results to the compute determinants of multiplicative commutators in $\mathcal{T}^\mu$.  

\begin{lemma}
Suppose that $A$ in $\mathcal{T}^\mu$ is invertible.  Then $A^{-1}$ is in $\mathcal{T}^\mu$.
\end{lemma}

\begin{proof}
Let $\phi = \sigma(A)$.  Then $A^{-1} = T_{\phi^{-1}} + K$ for some compact operator $K$; we must show that
$K$ is a trace class operator.  Because $I = AA^{-1} = AT_{\phi^{-1}} + AK$, we see that 
$AK = I -  AT_{\phi^{-1}}$ is trace class.  But the set of trace class operators is an ideal in the algebra of all
bounded operators on $H^2(S^1)$, so $K = A^{-1}(AK)$ is trace class.
\end{proof}

\begin{theorem}
Suppose that $A$ and $B$ are invertible elements of $\mathcal{T}^\mu$.  Then
\[
\det(ABA^{-1}B^{-1}) = \prod_{z \in \mathbb{D}} (\sigma(A), \sigma(B))_z^{-1}.
\]
\end{theorem}

\begin{proof}
Write $\phi = \sigma(A)$ and $\psi = \sigma(B)$.  Then
\[
R = \begin{pmatrix} A & 0 & 0 \\ 0 & A^{-1} & 0 \\ 0 & 0 & I \end{pmatrix} \qquad
\text{and} \qquad
S = \begin{pmatrix} B & 0 & 0 \\ 0 & I & 0 \\ 0 & 0 & B^{-1} \end{pmatrix}
\]
are lifts of $\Phi^{1,2}$ and $\Psi^{1,3}$ respectively.  Furthermore,
\[
RSR^{-1}S^{-1} =
\begin{pmatrix} ABA^{-1}B^{-1} & 0 & 0 \\ 0 & I & 0 \\  0 & 0 & I \end{pmatrix},
\]
so $\det(ABA^{-1}B^{-1}) = \det(RSR^{-1}S^{-1})$.  The desired result therefore follows from 
Theorem \ref{big theorem}.
\end{proof}

\begin{corollary}
Suppose that $\phi$ and $\psi$ are in $(C^\mu(S^1))^*$ and 
have winding number zero with respect to the origin.  Then
\[
\det(T_\phi T_\psi T_\phi^{-1} T_\psi^{-1})  = \prod_{z \in \mathbb{D}} (\phi, \psi)_z^{-1}.
\]
\end{corollary}

\end{document}